\documentclass{article}
\usepackage{amsfonts}
\usepackage{amsmath}

\newtheorem{theorem}{Theorem}

\newtheorem{definition}[theorem]{Definition}
\newtheorem{example}[theorem]{Example}

\newtheorem{lemma}[theorem]{Lemma}

\newtheorem{proposition}[theorem]{Proposition}

\newenvironment{proof}[1][Proof]{\noindent \textbf{#1.} }{\  \rule{0.5em}{0.5em}}
\input{tcilatex}
\begin{document}

\title{PSEUDO-RIEMANNIAN SUBMANIFOLDS WITH $3$-PLANAR GEODESICS }
\author{Kadri ARSLAN, Bet\"{u}l BULCA and G\"{u}nay \"{O}ZT\"{U}RK }
\maketitle

\begin{abstract}
In the present paper we study pseudo-Riemannian submanifolds which have
3-planar geodesic normal sections.We consider W-curves (helices) on
pseudo-Riemannian submanifolds. Finally, we give neccessary and sufficient
condition for a normal section to be a W-curve on pseudo-Riemannian
submanifolds.
\end{abstract}

\section{\textbf{Introduction}}

\footnote{%
2000 \textit{Mathematics Subject Classification}. 53C40, 53C42
\par
\textit{Key words and phrases}: Pseudo-Riemannian submanifold, geodesic
normal section.} In a Riemannian manifold, a regular curve is called a helix
if its first and second curvatures is constant and the third curvature is
zero. In 1980 Ikawa investigated the condition that every helix in a
Riemannian submanifold is a helix in the ambient space \cite{Ik1}. In a
pseudo-Riemannian manifold, helices are defined by almost the same way as
the Riemannian case. The same author also characterized the helices in
Lorentzian submanifold \cite{Ik2}.

An isometric immersion $f:M_{r}^{n}\rightarrow \mathbb{R}_{s}^{N}$ is said
to be planar geodesic if the image of each geodesic of $M_{r}$ lies in a $2$%
-plane of $\mathbb{R}_{s}^{N}.$ In the Riemannian case such immersions were
studied and classified by Hong \cite{Ho}, Little \cite{Li}, Sakamoto \cite%
{Sa}, Ferus \cite{Fe} and others. Further, Blomstrom classified planar
geodesic immersions with indefinite metric \cite{Bl}. It has been shown that
all parallel, planar geodesic surfaces in $\mathbb{R}_{s}^{N}$ are the
pseudo-Riemannian spheres, the Veronese surfaces and certain flat quadratic
surfaces. Recently Kim studied minimal surfaces of pseudo-Euclidean spaces
with geodesic normal sections. He proved that complete connected minimal
surfaces in a $5$-dimensional pseudo-Euclidean space with geodesic normal
sections are totally geodesics or flat quadrics \cite{Ki1}.

In the present work, we give some results toward a characterization of $3$%
-planar geodesic immersions $f:M_{r}\rightarrow \mathbb{N}_{s}$ from an $n$%
-dimensional, connected pseudo-Riemannian manifold $M_{r}$ into $m$%
-dimensional pseudo-Riemannian manifold $\mathbb{N}_{s}.$ Further, We
consider $W$-curves (helices) on pseudo-Riemannian submanifolds. Finally, we
give necessary and sufficient condition for a normal section to be a $W$%
-curve on pseudo-Riemannian submanifolds.

\section{\textbf{Basic Concepts}}

Let $f:M_{r}\rightarrow \mathbb{N}_{s}$ be an isometric immersion from an $n$%
-dimensional, connected pseudo-Riemannian manifold $M_{r}$ of index $r$ $%
(0\leq r\leq n)$ into $m$-dimensional pseudo-Riemannian manifold $\mathbb{N}%
_{s}$ of index $s$. Let $\nabla $ and $\widetilde{\nabla }$ denote the
covariant derivatives of $M_{r}$ and $\mathbb{N}_{s}$ respectively. Thus $%
\widetilde{\nabla }_{X}$ is just the directional derivative in the direction 
$X$ in $\mathbb{N}_{s}.$ Then for tangent vector fields $X$, $Y$ the \textit{%
second fundamental form} $h$ of the immersion $f$ is defined by 
\begin{equation}
h(X,Y)=\overset{\sim }{\nabla }_{X}Y-\nabla _{X}Y.  \tag{1.1}  \label{A1}
\end{equation}

For a vector field $\xi $ normal to $M_{r}$ we put 
\begin{equation}
\widetilde{\nabla }_{X}\xi =-A_{\xi }X+D_{X}\xi ,  \tag{1.2}  \label{A2}
\end{equation}%
where $A_{\xi }$ is the shape operator of $M_{r}$ and $D$ is the normal
connection of $M_{r}$. We have the following relation 
\begin{equation}
<A_{\xi }X,Y>=<h(X,Y),\xi >\text{.}  \tag{1.3}  \label{A3}
\end{equation}

The covariant derivatives of $h$ denoted respectively by $\overline{\nabla }%
h $ and $\overline{\nabla }$ $\overline{\nabla }h$ to be;

\begin{equation}
(\overline{\nabla }_{X}h)(Y,Z)=D_{X}h(Y,Z)-h(\nabla _{X}Y,Z)-h(Y,\nabla
_{X}Z),  \tag{1.4}  \label{A4}
\end{equation}%
and

\begin{eqnarray}
(\overline{\nabla }_{W}\overline{\nabla }_{X}h)(Y,Z) &=&D_{W}((\overline{%
\nabla }_{X}h)(Y,Z))-(\overline{\nabla }_{\nabla _{W}X}h)(Y,Z)-  \TCItag{1.5}
\label{A5} \\
&&-(\overline{\nabla }_{X}h)(\nabla _{W}Y,Z)-(\overline{\nabla }%
_{X}h)(Y,\nabla _{W}Z),  \notag
\end{eqnarray}%
where $X,Y$, $Z$ and $W$ are tangent vector fields over $M_{r}$ and $%
\overline{\nabla }$ is the Vander Waerden-Bortolotti connection \cite{Ch1}.
Then we obtain the Codazzi equation 
\begin{equation}
(\overline{\nabla }_{X}h)(Y,Z)=(\overline{\nabla }_{Y}h)(X,Z)=(\overline{%
\nabla }_{Z}h)(X,Y)\text{.}  \tag{1.6}  \label{A6}
\end{equation}

It is a well-known property that $\overline{\nabla }h$ is a trilinear
symmetric form on $M_{r}^{\text{ }}$ with values in the normal bundle $%
N(M_{r})$ and it is called the \textit{third fundamental form}. If $%
\overline{\nabla }h=0,$ then the second fundamental form is said to be%
\textit{\ parallel }\cite{FS}\textit{\ (i.e. }$M$ \textit{is 1-parallel }%
\cite{ALMO}\textit{)}. If $\overline{\nabla }$ $\overline{\nabla }h=0,$ then
the third fundamental form is said to be\textit{\ parallel }\cite{Lu}\textit{%
\ (i.e. }$M$\textit{\ is 2-parallel }\cite{ALMO}\textit{).}

The mean curvature vector field $H$ of $M_{r}$ is defined by 
\begin{equation}
H=\frac{1}{n}\sum <e_{i},e_{i}>h(e_{i},e_{i}),i=1,...,n.  \tag{1.7}
\label{A7}
\end{equation}
where $\left \{ e_{1},e_{2},...,e_{n}\right \} $ is an orthonormal frame
field of $M_{r}.$ $H$ is said to be parallel when $DH=0$ holds.

If the second fundamental form $h$ satisfies 
\begin{equation}
g(X,Y)H=h(X,Y),  \tag{1.8}  \label{A8}
\end{equation}%
for any tangent vector fields $X,Y$ of $M_{r},$ then $M_{r}$ is called a
totally umbilical. A totally umbilical submanifold with parallel mean
curvature vector fields is said to be an e\textit{xtrinsic sphere }\cite{Nak}%
.

\section{\textbf{Helices in a Pseudo-Riemannian Manifold}}

\smallskip Let $\gamma $ be a regular curve in a pseudo-Riemannian manifold $%
M_{r}.$ We denote the tangent vector field $\gamma ^{\prime }(s)$ by the
letter $X,$ when $\left \langle X,X\right \rangle =+1$ or $-1,$ $\gamma $ is
called a \textit{unit speed curve}. The curve $\gamma $ is called a \textit{%
Frenet curve of osculating order }$d$ (See \cite{FS}) if its derivatives $%
\gamma ^{^{\prime }}(s),\gamma ^{^{\prime \prime }}(s),...,\gamma ^{(d)}(s)$
are linearly independent and $\gamma ^{^{\prime }}(s),\gamma ^{^{\prime
\prime }}(s),...,\gamma ^{(d+1)}(s)$ are no longer linearly independent for
all $s\in I.$

To each Frenet curve of order $d$ we can associate an orthonormal $d$ frame $%
\left \{ V_{1},V_{2},...,V_{d}\right \} $ along $\gamma ,$called the \textit{%
Frenet frame}, and $k_{1},k_{2},...,k_{d-1}$ are \textit{curvature functions}
of $\gamma $.

\begin{proposition}
\cite{Mu}. If $\gamma :I\longrightarrow M_{r}$ is a non-null differentiable
curve of an $n$-dimensional pseudo-Riemannian manifold $M_{r}$ of osculating
order $d$ $(0\leq d\leq n)$ and $\left \{ V_{1}=X,V_{2},...,V_{d}\right \} $
is the Frenet frame of $\gamma $ then 
\begin{equation}
V_{1}^{^{\prime }}=\nabla _{X}X=\varepsilon _{2}k_{1}V_{2},  \tag{2.1}
\label{B1}
\end{equation}%
\begin{equation}
V_{2}^{^{\prime }}=\nabla _{X}V_{2}=-\varepsilon _{1}k_{1}V_{1}+\varepsilon
_{3}k_{2}V_{3},  \tag{2.2}  \label{B2}
\end{equation}%
\begin{equation*}
\vdots
\end{equation*}%
\begin{equation}
V_{d-1}^{^{\prime }}=\nabla _{X}V_{d-1}=-\varepsilon
_{(d-2)}k_{(d-2)}V_{(d-2)}+\varepsilon _{d}k_{(d-1)}V_{d},  \tag{2.3}
\label{B3}
\end{equation}%
\begin{equation}
V_{d}^{^{\prime }}=\nabla _{X}V_{d}=-\varepsilon _{(d-1)}k_{(d-1)}V_{(d-1)},
\tag{2.4}  \label{B4}
\end{equation}%
where $\varepsilon _{i}=\left \langle V_{i},V_{i}\right \rangle =\pm 1,\
k_{i}$, $1\leq i\leq (d-1)$ are curvature functions of $\gamma .$
\end{proposition}

\begin{definition}
Let $\gamma $ be a smooth curve of osculating order $d$ on $M_{r}.$ The
curve $\gamma $ is called a $\mathit{W}$\textit{-curve (or a helix)} of rank 
$d$ if $k_{1},k_{2},...,k_{d-1}$ are constant and $k_{d}=0.$ In particular,
a $W$-curve of rank $2$ is called a \textit{geodesics circle. A }$W$-curve
of rank $3$ is a \textit{right circular helix} \cite{FS}.
\end{definition}

\begin{proposition}
Let $\gamma $ be a non-null $W$-curve in $M_{r}$. If $\gamma $ is of rank $2$
then $\gamma ^{^{\prime \prime \prime }}$ is a scalar multiple of $\gamma
^{^{\prime }}.$ In this case necessarily%
\begin{equation}
\gamma ^{^{\prime \prime \prime }}(s)=-\varepsilon _{1}\varepsilon
_{2}k_{1}^{2}\gamma ^{^{\prime }}(s).  \tag{2.5}  \label{B5a}
\end{equation}
\end{proposition}

\begin{proof}
By the use of (2.1) we have $\gamma ^{^{\prime \prime }}(s)=\varepsilon
_{2}k_{1}V_{2}(s).$ Furthermore, differentiating this equation with respect
to $s$ and using (2.2) we obtain 
\begin{equation}
\gamma ^{^{\prime \prime \prime }}(s)=-\varepsilon _{1}\varepsilon
_{2}k_{1}^{2}X+\varepsilon _{2}k_{1}^{^{\prime }}V_{2}(s)+\varepsilon
_{2}\varepsilon _{3}k_{1}k_{2}V_{3}(s).  \tag{2.6}  \label{B8}
\end{equation}%
Since $\gamma $ is a W-curve of rank $2$ then by definition $k_{1}$ is
constant and $k_{2}=0$ we get the result.
\end{proof}

\begin{proposition}
Let $\gamma $ be a non-null $W$-curve of $M_{r}$. If $\gamma $ is of
osculating order $3$ then 
\begin{equation}
\gamma ^{^{\prime \prime \prime \prime }}(s)=-\varepsilon _{2}(\varepsilon
_{1}k_{1}^{2}+\varepsilon _{3}k_{2}^{2})\gamma ^{^{\prime \prime }}(s). 
\tag{2.7}  \label{B9}
\end{equation}
\end{proposition}

\begin{proof}
Differentiating (2.6) and using the fact that $k_{1},k_{2}$ are constant and 
$k_{3}=0$ we get the result.
\end{proof}

\section{\textbf{Planar Geodesic Immersions}}

Let $f:M_{r}\rightarrow \mathbb{N}_{s}$ be an isometric immersion from an $n$%
-dimensional, connected pseudo-Riemannian manifold $M_{r}$ of index $r$ $%
(0\leq r\leq n)$ into $m$-dimensional pseudo-Riemannian manifold $\mathbb{N}%
_{s} $ of index $s.$ For a point $p\in M_{r}$ and a unit vector $X\in
T_{p}(M_{r}) $ the vector $X$ and the normal space $T_{p}^{\bot }(M_{r})$
determine a $(m-n+1)$-dimensional subspace $E(p,X)$ of $T_{f(p)}(N_{s})$
which determines a $(m-n+1)$-dimensional totally geodesic submanifold $W$ of 
$\mathbb{N}_{s}$. The intersection of $M_{r}$ with $W$ gives rise a curve $%
\gamma $ (in a neighborhood of $p$) called the \textit{normal section} of $%
M_{r}$ at point $p$ in the direction of $X$ \cite{Ch2}.

The submanifold $M_{r}$ (or the isometric immersion $f$) is said to have $d$%
\textit{-planar normal sections} if for each normal section $\gamma $ the
first, second and higher order derivatives $\gamma ^{^{\prime }}(s),\gamma
^{^{\prime \prime }}(s),...,\gamma ^{(d)}(s),\gamma ^{(d+1)}(s)$ ;$(1\leq
d\leq m-n+1)$ are linearly dependent as vectors in $W$\cite{Ch2}.

The submanifold $M_{r}$ is said to have\textit{\ }$d$\textit{-planar
geodesic normal sections} if each normal section of $M_{r}$ is a geodesic of 
$M_{r}.$

In \cite{Bl} the immersion in pseudo-Euclidean space with $2$-planar
geodesic normal section have been studied by Blomstrom(See also \cite{Ho}).

We have the following result.

\begin{proposition}
Let $\gamma $ be a non-null geodesic normal section of $M_{r}.$ If $\gamma
^{\prime }(s)=X(s)$, then we have 
\begin{equation}
\gamma ^{^{\prime \prime }}(s)=h(X,X),  \tag{3.1}  \label{C1}
\end{equation}%
\begin{equation}
\gamma ^{^{^{\prime \prime \prime }}}(s)=-A_{h(X,X)}X+(\overline{\nabla }%
_{X}h)(X,X),  \tag{3.2}  \label{C.2}
\end{equation}%
\begin{equation}
\begin{array}{c}
\gamma ^{^{\prime \prime \prime \prime }}(s)=-\nabla
_{X}(A_{h(X,X)}X)-h(A_{h(X,X)}X,X) \\ 
\text{ \  \  \  \  \  \  \  \  \  \  \  \  \ }-A_{(\overline{\nabla }_{X}h)(X,X)}X+(%
\overline{\nabla }_{X}\overline{\nabla }_{X}h)(X,X).%
\end{array}
\tag{3.3}  \label{C.3}
\end{equation}
\end{proposition}

\begin{example}
\cite{Bl} Pseudo-Riemannian sphere%
\begin{equation}
S_{r}^{n}(c)=\left \{ p\in \mathbb{E}_{r}^{n+1}:<p-a,p-a>=\frac{1}{c}\right
\} ,c>0,  \tag{3.4}  \label{C4}
\end{equation}%
and pseudo-Riemannian hyperbolic space%
\begin{equation}
H_{r}^{n}(c)=\left \{ p\in \mathbb{E}_{r+1}^{n+1}:<p-a,p-a>=\frac{1}{c}%
\right \} ,c<0,  \tag{3.5}  \label{C5}
\end{equation}%
both have $2$-planar geodesic normal sections.
\end{example}

\begin{definition}
The submanifold $M_{r}$ (or the isometric immersion $f$) is said to be%
\textit{\ pseudo-isotropic} at $p$ if 
\begin{equation*}
L=<h(X,X),h(X,X)>,
\end{equation*}%
is independent of the choice of unit vector $X$ tangent to $M_{r}$ at $p$.
In particular if $L$ is independent of the points then $M_{r}$ is said to be
constant pseudo-isotropic.

The submanifold $M_{r}$ is pseudo-isotropic if and only if\textit{\ }%
\begin{equation*}
<h(X,X),h(X,Y)>=0,
\end{equation*}%
for any orthonormal vectors $X$ and $Y$ \cite{Bl}.
\end{definition}

The following results are well-known.

\begin{theorem}
\cite{Bl}. If the immersion $f:M_{r}^{2}\rightarrow \mathbb{E}_{s}^{m}$ has $%
2$-planar geodesic normal sections, then $f(M)$ is a submanifold with zero
mean curvature in a hypersphere $S_{s-1}^{m-1}$ or $H_{s-1}^{m-1}$ if and
only if $L$ is a non-zero constant.
\end{theorem}

\begin{theorem}
\cite{Ki1}. The immersion $f:M_{r}^{2}\rightarrow \mathbb{E}_{s}^{m}$ with $%
2 $-planar geodesic normal sections is constant pseudo-isotropic.
\end{theorem}

\begin{theorem}
\cite{Ki2}. Let $M_{r}$ be a pseudo-Riemannian submanifold of index $r$ of a
pseudo-Euclidean space $\mathbb{E}_{s}^{m}$ of index $s$ with geodesic
normal sections. Then 
\begin{equation}
\left \langle (\overline{\nabla }_{X}h)(X,X),(\overline{\nabla }%
_{X}h)(X,X)\right \rangle ,  \tag{3.6}  \label{C6}
\end{equation}%
is constant on the their tangent bundle $UM$ of $M_{r}.$
\end{theorem}

\begin{theorem}
\cite{Ki2}. Let $M_{r}$ be a minimal surface of $\mathbb{E}_{s}^{5}$ with
geodesics normal sections. Then we have

$i)$\textbf{\ }$M_{r}$ is $1$-parallel and $0$-pseudo isotropic (i.e. $L=0$),

$ii)$\textbf{\ }$M_{r}$ has $2$-planar geodesic normal sections,

$iii)$\textbf{\ }$M_{r}$ is flat.
\end{theorem}

Submanifolds $M$ in $\mathbb{R}^{n+d}$ with $3$-planar normal sections have
been studied by S.J.Li for the case $M$ is isotropic \cite{Li1} and sphered 
\cite{Li2}. See also \cite{AW} for the case $M$ is a product manifold in $%
\mathbb{R}^{n+d}.$ In \cite{AC} the authors consider submanifolds in a real
space form $\mathbb{N}^{n+d}(c)$ with $3$-planar geodesic normal sections.

We proved the following results$.$

\begin{lemma}
Let $f:M_{r}\rightarrow \mathbb{N}_{s}$ be an isometric immersion with $3$%
-planar geodesic normal sections then $f$ is constant pseudo-isotropic.
\end{lemma}

\begin{proof}
Similar to the proof of Lemma 4.1 in \cite{Na}.
\end{proof}

\begin{proposition}
Let $f:M_{r}\rightarrow \mathbb{N}_{s}$ be an isometric immersion with $3$%
-planar geodesic normal sections then we have 
\begin{equation}
(\overline{\nabla }_{X}h)(X,X)=\varepsilon _{2}(Xk_{1})V_{2}+\varepsilon
_{2}\varepsilon _{3}k_{1}k_{2}V_{3},  \tag{3.7}  \label{C7}
\end{equation}%
\begin{equation}
A_{h(X,X)}X=\varepsilon _{1}\varepsilon _{2}k_{1}^{2}X.  \tag{3.8}
\label{C8}
\end{equation}
\end{proposition}

\proof%
Let $\gamma $ be a normal section of $M_{r}$ at point $p=\gamma (s)$ in the
direction of $X$. We suppose that $k_{1}(s)$ is positive. Then $k_{1}$ is
also smooth and there exists a unit vector field $V_{2}$ along $\gamma $
normal to $M_{r}$ such that 
\begin{equation}
h(X,X)=\left \langle V_{2},V_{2}\right \rangle k_{1}V_{2}.  \tag{3.9}
\label{C9}
\end{equation}

Since $\overline{\nabla }_{X}V_{2}$ is also tangent to $M_{r}$, there exists
a vector field $V_{3}$ normal to $M_{r}$ and mutually ortogonal to $X$ and $%
V_{2}$ such that 
\begin{equation}
\widetilde{\nabla }_{X}V_{2}=-\left \langle X,X\right \rangle k_{1}X+\left
\langle V_{3},V_{3}\right \rangle k_{2}V_{3}.  \tag{3.10}  \label{C10}
\end{equation}

Differentiating (3.9) covariantly and using (3.10) we get 
\begin{equation}
(\overline{\nabla }_{X}h)(X,X)=-\varepsilon _{1}\varepsilon
_{2}k_{1}^{2}X+\varepsilon _{2}(Xk_{1})V_{2}+\varepsilon _{2}\varepsilon
_{3}k_{1}k_{2}V_{3},  \tag{3.11}  \label{C11}
\end{equation}%
where $\left \langle V_{i},V_{i}\right \rangle =\varepsilon _{i}=\pm 1.$
Comparing (3.11) with (3.2) we get the result.

\begin{proposition}
Let $\gamma $ be a normal section of $M_{r}$ at point $p=\gamma (s)$ in the
direction of $X.$ $\gamma $ is a non-null $W$-curve of rank $2$ in $M_{r}$
if and only if 
\begin{equation}
\nabla _{X}\nabla _{X}X+g(\nabla _{X}X,\nabla _{X}X)g(X,X)X=0.  \tag{3.12}
\label{C12}
\end{equation}
\end{proposition}

\begin{proof}
Since $\gamma ^{\prime }(s)=X(s),$ $\gamma ^{^{\prime \prime }}(s)=\nabla
_{X}\nabla _{X}X$ and 
\begin{equation*}
g(X,X)=\varepsilon _{1},g(\nabla _{X}X,\nabla _{X}X)=\varepsilon
_{2}k_{1}^{2}.
\end{equation*}%
So, by the use of the equality $\gamma ^{^{\prime \prime }}(s)=\varepsilon
_{2}k_{1}V_{2}(s)$ we get the result.
\end{proof}

\begin{proposition}
Let $M_{r}$ be a totally umbilical submanifold of $\mathbb{N}_{s}$ with
parallel mean curvature vector field . If the normal section $\gamma $ is a
W- curve of osculating order 2. Then $\gamma $ is also a W-curve of $\mathbb{%
N}_{s}$ with the same order.
\end{proposition}

\begin{proof}
Suppose $\gamma $ is a W-curve of rank $2$ in $M_{r}$ then it satisfies the
equality (3.12). Further, by the use of (1.1) we get 
\begin{equation}
\gamma ^{^{\prime \prime }}=\widetilde{\nabla }_{X}X=\nabla _{X}X+h(X,X). 
\tag{3.13}  \label{C13}
\end{equation}%
Since $M_{r}$ is totally umbilical then $g(X,X)H=h(X,X).$ So, the equation
(3.13) reduces to 
\begin{equation}
\gamma ^{^{\prime \prime }}=\widetilde{\nabla }_{X}X=\nabla _{X}X+g(X,X)H. 
\tag{3.14}  \label{C14}
\end{equation}%
Differentiating the equation (3.14) with respect to X we obtain%
\begin{eqnarray}
\gamma ^{^{\prime \prime \prime }} &=&\widetilde{\nabla }_{X}\widetilde{%
\nabla }_{X}X=\nabla _{X}\nabla _{X}X+g(X,\nabla _{X}X)H  \TCItag{3.15}
\label{C15} \\
&&+g(X,X)(-A_{H}X+D_{X}H).  \notag
\end{eqnarray}%
Further, taking use of $DH=0$ and (3.13)-(3.15) get 
\begin{eqnarray*}
&&\widetilde{\nabla }_{X}\widetilde{\nabla }_{X}X+g(\widetilde{\nabla }_{X}X,%
\widetilde{\nabla }_{X}X)g(X,X)X \\
&=&\nabla _{X}\nabla _{X}X-g(H,H)g(X,X)X+\left \{ g(\nabla _{X}X,\nabla
_{X}X)g(X,X)\right \} g(X,X)X \\
&=&\nabla _{X}\nabla _{X}X+g(\nabla _{X}X,\nabla _{X}X)g(X,X)X.
\end{eqnarray*}%
So, by previous proposition $\gamma $ is a W-curve of rank $2$ in $\mathbb{N}%
_{s}.$
\end{proof}

\emph{Kadri Arslan \& Bet\"{u}l BULCA}

\emph{Uludag University}

\emph{Faculty of Art and Sciences}

\emph{Department of Mathematics}

\emph{16059, Bursa, TURKEY.}

\emph{arslan@uludag.edu.tr}

\emph{bbulca@uludag.edu.tr}

\bigskip

\emph{G\"{u}nay \"{O}ZT\"{U}RK }

\emph{Kocaeli University}

\emph{Faculty of Art and Sciences}

\emph{Department of Mathematics}

\emph{41310, Kocaeli, TURKEY.}

\emph{ogunay@kocaeli.edu.tr}

\end{document}